\documentclass[reqno]{amsart}
\usepackage{graphicx}
\usepackage{epsfig}
\usepackage{amssymb}
\usepackage{xcolor}
\usepackage{amsfonts}
\usepackage{dsfont}
\usepackage{amsmath}
\usepackage{stmaryrd}
\usepackage{hyperref}
\usepackage{subfigure, verbatim}
\newtheorem{theorem}{Theorem}
\newtheorem{lemma}{Lemma}
\theoremstyle{proposition}
\newtheorem{proposition}{Proposition}
\theoremstyle{definition}

\theoremstyle{corollary}
\newtheorem{cor}{Corollary}

\theoremstyle{remark}

\theoremstyle{remark}
\newtheorem{remark}[theorem]{Remark}

\newcommand{\beq}{\begin{equation}}
\newcommand{\eeq}{\end{equation}}
\newcommand{\beqn}{\begin{equation*}}
\newcommand{\eeqn}{\end{equation*}}
\newcommand{\bea}{\begin{eqnarray}}
\newcommand{\eea}{\end{eqnarray}}
\newcommand{\bean}{\begin{eqnarray*}}
\newcommand{\eean}{\end{eqnarray*}}

\newcommand{\be}{\begin{enumerate}}
\newcommand{\ee}{\end{enumerate}}
\newcommand{\bi}{\begin{itemize}}
\newcommand{\ei}{\end{itemize}}
\newcommand{\bd}{\begin{description}}
\newcommand{\ed}{\end{description}}

\title{The Fourier coefficients of the critical holomorphic multiplicative chaos}
 \author{Christopher Atherfold and Joseph Najnudel}

\begin{document}

\maketitle
\begin{abstract}
The holomorphic multiplicative chaos (HMC) is a holomorphic
analogue of the Gaussian multiplicative chaos. It arises naturally as the
limit in large matrix size of the characteristic polynomial of Haar unitary
matrices, and more generally, random matrices following the Circular-$\beta$-Ensemble. In \cite{NPS23}, Najnudel, Paquette and Simm prove that in the $L^2$ phase $\beta > 4$, the appropriately normalized Fourier coefficient of the HMC converges in distribution to the square root of the total mass of the Gaussian multiplicative chaos on the unit circle, multiplied by an independent complex normal random variable. 
This convergence has been extended to the $L^1$ phase by Najnudel, Paquette, Simm and Vu in \cite{NPSV25}. In the present article, we prove 
that this convergence further extends to the critical case $\beta = 2$, which corresponds to the limiting coefficients of the characteristic polynomial of the Circular Unitary Ensemble. We also prove the joint convergence of consecutive Fourier coefficients, and
we derive convergence in distribution of the secular coefficients of the Circular Unitary Ensemble with index growing sufficiently slowly with the dimension. 
\end{abstract}

\section{Introduction}
The holomorphic multiplicative chaos (HMC) is a random distribution on the circle, arising in the theory of random matrices and in analytic number theory, and introduced by Najnudel, Paquette and Simm in \cite{NPS23}. 
The HMC is defined as follows. Let $(\mathcal{N}_k)_{k \geq 1}$ be i.i.d. complex Gaussian random variables, such that 
$$ \mathbb{E}[ \mathcal{N}_k] = 
\mathbb{E}[ \mathcal{N}_k^2] = 0 
 \quad
  \text{and}
  \quad
  \mathbb{E}[|\mathcal{N}_k|^2]= 1. $$
Define $G^{\mathbb{C}}$ be the Gaussian holomorphic function on the open unit disc, given by the power series 
$$ 
  G^{\mathbb{C}}(z) = \sum_{k=1}^\infty \frac{z^k}{\sqrt{k}}\mathcal{N}_k.
$$

For a positive parameter $\theta > 0$, we define
the random distribution $\mathrm{HMC}_\theta$
by its action on trigonometric polynomials on the unit circle,
which extend harmonically to $\mathbb{C}^*$ as 
Laurent polynomials $\phi$: 
$$ 
  ( \mathrm{HMC}_\theta, \phi)
  = 
  \lim_{r \to 1} 
  \frac{1}{2\pi}
  \int_0^{2\pi} 
  e^{\sqrt{\theta} G^{\mathbb{C}}(re^{i\vartheta})}
  \overline{\phi(r\vartheta)} d\vartheta.
$$
The random distribution $\mathrm{HMC}_\theta$ is, by definition, the 
\emph{holomorphic multiplicative chaos}
of parameter $\theta$. 
For an integer $n$, we define $c_n$ as the 
 the $n$-th Fourier coefficient of the HMC: 
$$
  c_n =  ( \mathrm{HMC}_\theta, \vartheta \mapsto  e^{i n\vartheta}).
$$
For $n < 0$, $c_n$ is equal to zero, so from now, we always assume $n \geq 0$. 
In this case, the coefficient $c_{n}$ is the 
degree $n$ coefficient of the series $G^{\mathbb{C}}$: 
$$
c_{n} = [z^{n}]\,e^{\sqrt{\theta}G^{\mathbb{C}}(z)} = [z^{n}]\,\mathrm{exp}\left(\sqrt{\theta}\sum_{k=1}^{\infty}\frac{z^{k}}{\sqrt{k}}\,\mathcal{N}_{k}\right),
$$
where $[z^{n}]\,h(z)$ denotes the coefficient of $z^{n}$ in the power series expansion of $h(z)$ around the point $z=0$. 

As a consequence of results by Diaconis and Shahshahani \cite{DS94} in the case $\theta =1$, 
and by Jiang and Matsumoto \cite{JM15} in the general case, the random holomorphic function  
$$z \mapsto \mathrm{exp}\left(\sqrt{\theta}\sum_{k=1}^{\infty}\frac{z^{k}}{\sqrt{k}}\,\mathcal{N}_{k}\right)$$
is the limiting distribution of the characteristic polynomial of the Circular-$\beta$-Ensemble inside the open unit disc, where the inverse temperature parameter $\beta$ is given 
by $$\theta = \frac{2}{\beta}.$$
The case $\theta = 1$, $\beta = 2$ corresponds to the characteristic polynomial of the Circular Unitary Ensemble. 
The family of coefficients $(c_n)_{n \geq 0}$ is the limit in distribution, for the convergence of finite-dimensional marginals, of the family of  coefficients of the characteristic polynomial of Circular-$\beta$-Ensembles when the dimension tends to infinity. These coefficients are also called secular coefficients.  

In the work of Soundararajan--Zaman \cite{SZ22}, they identified that these secular coefficients can also be obtained from the large $q$-limit of the function field analogue to the work of Harper in \cite{H20}, using the convergence in distribution of 
$$\left(\widetilde{X}(k) := \frac{\sqrt{k}}{q^{k/2}}\sum_{\substack{\text{deg}(P)| k \\r=\frac{\text{deg}(P)}{k}}}\frac{(f(P))^r}{r} \right)_{k \geq 1}$$
to independent complex Gaussian variables when $q \rightarrow \infty$, $P$ being a monic irreducible polynomial on the field $\mathbb{F}_q$ with $q$ elements, the variables $f(P)$ being independent, uniform on the complex unit circle. This has been studied in a couple of papers recently, for example in \cite{agg22}, \cite{aggar22} and \cite{hof24}. 

The coefficients $c_n$ were studied in detail in \cite{NPS23}, where a result of convergence in distribution is shown for $0 < \theta < 1/2$,
and extended to $0 < \theta < 1$ in \cite{NPSV25}. In \cite{GW24} and \cite{GW25}, Gorodetsky and Wong establish analogues of this convergence for a class of random multiplicative
functions.
In the present article, we extends the main result of \cite{NPSV25} on the coefficients $c_n$ to the critical case $\theta = 1$, which corresponds to $\beta =2$, i.e. the Circular Unitary Ensemble.
Similar extension to critical setting has been very recently obtained in a remarkable preprint by Gorodetsky and Wong in \cite{GW25+}. While some of the core ideas are shared, these works were developed independently of each other. 
The case $\theta = 1$ is different from the case $0 < \theta < 1$ because of the phenomenon 
of \textit{better than square root cancellation}, which appears both in the present HMC setting and in the setting of analytic number theory. As already 
seen in \cite{NPS23}, the naive bound 
$$\mathbb{E}[|c_n|] \leq (\mathbb{E}[ |c_n|^2])^{1/2},$$
which gives the true order of magnitude 
of the expectation of $|c_n|$ when $0 <\theta < 1$,
overestimates this order of magnitude 
when $\theta = 1$. In this case, the naive bound
is equal to $1$, whereas the exact order of magnitude 
is $(\log n)^{-1/4}$, as proven in \cite{SZ22}.
The better than square root cancellation has been previously proven in the setting of random multiplicative functions by
Harper in \cite{H20}. 

The limiting distribution obtained for the coefficients $c_n$ involves the  
 Gaussian multiplicative chaos (GMC), which in 
 the critical case $\theta = 1$, can be 
 defined as follows:  
\begin{align*}
  \mathrm{GMC}_1(d\vartheta) & = 
  \lim_{r \in (0,1), r \to 1} (1-r^2)
  \sqrt{\log ((1-r^2)^{-1})} | e^{\sqrt{\theta}G^{\mathbb{C}}(r e^{i\vartheta})}|^2  \frac{d\vartheta}{2 \pi}
\\ &   =\lim_{r \in (0,1), r \to 1} (1-r^2)  \sqrt{\log ((1-r^2)^{-1})} e^{\sqrt{\theta}G(r e^{i\vartheta})} \frac{d\vartheta}{2 \pi} 
  \end{align*}
  where $G(z) := 2 \Re G^{\mathbb{C}} (z)$
  for $|z| < 1$. 
The existence of this limit as a random measure is
a particular case of our Proposition \ref{GMCcritical} and is directly deduced from results by Junnila and Saksman \cite{JS17}. 

Moreover, it is shown by Remy in \cite{R20} that the total mass of this particular random measure is distributed as the inverse of an exponential random variable. 
 
The main goal of the present article is to prove the following result:
\begin{theorem}  \label{main}
 In the critical case $\theta = 1$, we have the convergence in distribution: 
  \begin{equation}
    c_n (\log n)^{1/4} \underset{n \rightarrow 
    \infty}{\longrightarrow} \sqrt{\mathcal{M}_{1}}\mathcal{Z} \label{l1phaselim}
  \end{equation}
  where $\mathcal{Z}$ and $\mathcal{M}_{1}$ are independent, $\mathcal{Z}$ being a
  complex Gaussian variable such that 
  $$\mathbb{E} [ \mathcal{Z} ] = 
  \mathbb{E} [ \mathcal{Z}^2], \; \; 
  \mathbb{E} [ |\mathcal{Z}|^2] = 1,$$
and $\mathcal{M}_{1}$ is the
  total mass of the random measure $ \mathrm{GMC}_1$.
  Moreover, $\mathcal{M}_1$ is the inverse of an exponential variable of mean $\sqrt{\pi}$. Therefore, $\sqrt{\pi} |c_n|^2 (\log n)^{1/2}$ converges to the ratio of two independent standard exponential variables, which has density $x \mapsto (1+x)^{-2}$ with respect to the Lebesgue measure. 
\end{theorem}
As mentioned previously, this gives a description of the distribution of the model in \cite{SZ22} as well.

This result has a consequence on the secular coefficients of the characteristic polynomial of the Circular Unitary Ensemble:
\begin{cor}
For $N \geq 1$, let $U_N$ be a Haar distributed 
random matrix on the unitary group $U(N)$. 
For $0 \leq n \leq N$, let $c^{(N)}_n$ be the coefficient in $z^n$ of the characteristic polynomial $\operatorname{det} (I_N - z U_N)$. 
Then for any sequence $(N_n)_{n \geq 1}$ of integers such that $N_n \geq n$ and 
$$\frac{N_n}{ n \sqrt{\log n} \log \log n} 
\underset{n \rightarrow \infty}{\longrightarrow} 
\infty,$$
we have 
$$ c^{(N_n)}_n (\log n)^{1/4} \underset{n \rightarrow 
    \infty}{\longrightarrow} \sqrt{\mathcal{M}_{1}}\mathcal{Z} $$
    in distribution. 
\end{cor}

 \begin{proof}
By \cite{NPS23}, Lemma 7.2., we can couple 
$c_n$ with $c^{(N_n)}_n$ in such a way that 
$$\mathbb{E} [ (\log n)^{1/2}|c_n - c^{(N_n)}_n|^2] 
= \mathcal{O} \left( (\log n)^{1/2}\frac{n}{N_n} \log \left(\frac{N_n}{n} \right) \right)$$
for $n$ large enough, in order to ensure 
$n \leq N_n/2$. For all $C > 1$, there exists $n_C \geq 100$ such that 
for all $n \geq n_C$, 
$$\frac{N_n}{n} \geq C \sqrt{\log n} \log \log n
\geq e,$$
and then, since $(\log x)/x$ is decreasing in $x \geq e$, 
$$\frac{n}{N_n} \log \left(\frac{N_n}{n}\right)
\leq \frac{ \log C + (1/2) \log \log n + \log \log \log n}{C \sqrt{\log n} \log \log n}.$$
Hence,
$$\mathbb{E} [ (\log n)^{1/2}|c_n - c^{(N_n)}_n|^2]
= \mathcal{O} \left(  \frac{ \log C + (1/2) \log \log n + \log \log \log n}{C \log \log n}\right)  $$
Taking the upper limit when $n \rightarrow \infty$ and then letting $C \rightarrow \infty$ gives the convergence 
$$\mathbb{E} [ (\log n)^{1/2} |c_n - c_n^{(N_n)}|^2] \underset{n \rightarrow \infty}{\longrightarrow} 
0.$$
In particular, $(\log n)^{1/4} (c_n - c^{(N_n)}_n)$ converges in probability to zero, which proves the corollary by Slutsky's theorem. 
 \end{proof}

The main theorem can also be extended to finite-dimensional marginals of $(c_{n+k})_{k \in \mathbb{Z}}$. For a given complex polynomial 
$$p : z \mapsto \sum_{j=0}^d a_r z^r,$$ we define 
$$X_n (p) := \sum_{r=0}^d a_r c_{n+r}.$$ 
We then have the following: 
\begin{theorem} \label{mainmultidim}
In the critical case $\theta = 1$, the following convergence in distribution holds:
$$X_n (p) (\log n)^{1/4} \underset{n \rightarrow \infty}{\longrightarrow} \left( \int_0^{2 \pi} |p(e^{-i \vartheta})|^2 \mathrm{GMC}_1 (d \vartheta) \right)^{1/2} \mathcal{Z},$$
where 
$\mathcal{Z}$ is complex Gaussian, independent of $\mathrm{GMC}_1$, and 
 $$\mathbb{E} [ \mathcal{Z} ] = 
  \mathbb{E} [ \mathcal{Z}^2], \; \; 
  \mathbb{E} [ |\mathcal{Z}|^2] = 1.$$

\end{theorem}
As in \cite{NPSV25}, we deduce the following: 
\begin{cor}
Fix and integer $\ell \geq 1$, and define
the matrix $(\mathcal{H}_{k_1, k_2})_{0 \leq k_1, k_2 \leq \ell}$ by  
$$ \mathcal{H}_{k_1, k_2} := \int_0^{2 \pi} 
e^{ i \vartheta (k_2-k_1)}  \mathrm{GMC}_1 (d \vartheta).
$$
Then, 
$$ ((\log n)^{1/4} c_{n+k})_{0 \leq k \leq \ell} 
\underset{n \rightarrow \infty}{\longrightarrow} 
\sqrt{\mathcal{H}} (\mathcal{Z}_k)_{0 \leq k \leq \ell}
$$
where $(\mathcal{Z}_k)_{1 \leq k \leq \ell}$
are i.i.d. complex normal variables, independent of $\mathrm{GMC}_1$, with
 $$\mathbb{E} [ \mathcal{Z}_k ] = 
  \mathbb{E} [ \mathcal{Z}_k^2], \; \; 
  \mathbb{E} [ |\mathcal{Z}_k|^2] = 1.$$ \label{toeplitz_dist}
\end{cor}
\begin{proof}
It is sufficient to prove, 
for all reals $(b_k)_{0 \leq k \leq \ell}$, $(b'_k)_{0 \leq k \leq \ell}$, that 
$$(\log n)^{1/4} 
\left(\sum_{j=0}^{\ell} 
(b_{j} \Re c_{n+j} + b'_{j} \Im c_{n+j}) \right)$$
converges in distribution to 
$$\sum_{j=0}^{\ell} \left(b_j \Re \left( (\sqrt{\mathcal{H}} (\mathcal{Z}_k)_{0 \leq k \leq \ell} )_j \right) + 
b'_j \Im \left( (\sqrt{\mathcal{H}} (\mathcal{Z}_k)_{0 \leq k \leq \ell} )_j \right) \right)
$$
Letting  $a_k := b_k - i b'_k$, 
this convergence is equivalent to 
$$(\log n)^{1/4} 
 \left(\sum_{j=0}^{\ell} 
\Re (a_j c_{n+j}) \right) \underset{n \rightarrow \infty}{\longrightarrow} 
\sum_{j=0}^{\ell} 
\Re \left(a_j (\sqrt{\mathcal{H}} (\mathcal{Z}_k)_{0 \leq k \leq \ell} )_j\right) 
,$$
in distribution.
It is then enough to check the convergence in distribution 
$$X_n\left(z \mapsto \sum_{j=0}^\ell a_j  z^j \right) (\log n)^{1/4}
\underset{n \rightarrow \infty}{\longrightarrow}
\sum_{j=0}^{\ell} 
a_j (\sqrt{\mathcal{H}} (\mathcal{Z}_k)_{0 \leq k \leq \ell} )_j.
$$
By Theorem \ref{mainmultidim}, 
it is sufficient to check the following equality in distribution 
$$\left( \int_0^{2 \pi} \left|\sum_{j=0}^{\ell} a_j e^{-i j \vartheta}\right|^2 \mathrm{GMC}_1 (d \vartheta) \right)^{1/2} \mathcal{Z}
= \sum_{j=0}^{\ell} 
a_j (\sqrt{\mathcal{H}} (\mathcal{Z}_k)_{0 \leq k \leq \ell} )_j. $$
The right-hand side can be rewritten as 
$$\sum_{0 \leq j, k \leq \ell}
a_j (\sqrt{\mathcal{H}})_{j,k}
\mathcal{Z}_k
= \sum_{k=0}^{\ell} \left(\sum_{j=0}^{\ell} a_j (\sqrt{\mathcal{H}})_{j,k} \right) \mathcal{Z}_k,
$$
and then the two sides are complex Gaussian variables, both with expectation and expectation of the square equal to zero. 
It is then enough to check the equality of the expectation of their squared-moduli, i.e. 
$$\int_0^{2 \pi} \left|\sum_{j=0}^{\ell} a_j e^{-i j \vartheta}\right|^2 \mathrm{GMC}_1 (d \vartheta) 
= \sum_{k=0}^{\ell} 
\left| \sum_{j=0}^{\ell} 
a_j (\sqrt{\mathcal{H}})_{j,k}  \right|^2.$$
Expanding the squares and comparing the 
coefficients in $a_{j}\overline{a_{j'}}$
for $0 \leq j, j' \leq \ell$, we deduce that it is enough to check
$$\int_0^{2 \pi} 
e^{i \vartheta(j'-j)}\mathrm{GMC}_1 (d \vartheta)
= \sum_{k= 0}^{\ell} (\sqrt{\mathcal{H}})_{j,k}\overline{(\sqrt{\mathcal{H}})_{j',k}}.
$$
By definition, the left-hand side is 
$\mathcal{H}_{j,j'}$, whereas the right-hand side is, since $\sqrt{\mathcal{H}}$ is Hermitian: 
$$\sum_{k= 0}^{\ell} (\sqrt{\mathcal{H}})_{j,k}(\sqrt{\mathcal{H}})_{k,j'}
= ((\sqrt{\mathcal{H}})^2)_{j,j'}
= \mathcal{H}_{j,j'}.$$
\end{proof}
As a consequence of the main theorem, we also prove some asymptotics of low moments and tails of $|c_n|$. For example, we show that 
$$\mathbb{E}((\log(n))^{1/4}|c_n|) \underset{n \rightarrow \infty}{\longrightarrow}  \frac{\pi^{3/4}}{2}$$
(which answers a question in \cite{agg22}), and that for $y$ fixed,
$$\mathbb{P}\left((\log(n))^{1/4}|c_n| > y\right) \underset{n \rightarrow \infty}{\longrightarrow}  \frac{1}{1+y^2\sqrt{\pi}}.$$
In particular, we can identify the tail asymptotic conjectured in \cite{H20} in the context of random multiplicative functions. This is likely a universal behaviour and one would expect this to appear when computing the tails for sums of random multiplicative functions. This task is significantly simpler in the present setting, due to the closed form of the multiplicative chaos we have here compared to the random multiplicative functions case. We thank Christian Webb for this observation. 
\par
The present article is structured as follows.  
In Section \ref{approxchaos}, we provide a construction of the critical Gaussian multiplicative chaos on the unit circle which mixes truncation of Fourier series and harmonic extension of the logarithmically correlated field inside the unit disc. In Section \ref{splitting}, we split the expression of the coefficients $c_n$ into two sums, and show that Theorem \ref{main} and Theorem \ref{mainmultidim} are consequences of Propositions \ref{badsum} and \ref{goodsum}, each of these propositions concerning one of the two sums. 
Proposition \ref{badsum} is proven in Section \ref{proofbad}. In Section \ref{martingaleCLT}, we show, by using a suitable version of the martingale central limit theorem, that Proposition \ref{goodsum} is itself of consequence of 
two other propositions, which are proven in Section \ref{proofconvergencechaospsi} and Section \ref{prooflimitDickman}. The proof in Section \ref{proofconvergencechaospsi} uses the construction of the critical Gaussian multiplicative chaos given in Section \ref{approxchaos}. Finally, in Section 
\ref{asymptotics}, the asymptotics above on the moments and tails of the distribution
of $c_n$ are proven.
\section{Approximations of the critical Gaussian multiplicative chaos}
\label{approxchaos}
Starting from a logarithmically correlated field on the unit circle, two of most natural constructions of the Gaussian multiplicative chaos are obtained either by considering the harmonic extension of the field inside the unit disc or by truncating its Fourier series. In the present article, we need to use an approximation of the field mixing these two constructions. 
The precise result we need is the following, deduced from results by Junnila and Saksman in \cite{JS17}:

\begin{proposition} \label{GMCcritical}
For $n \in [1, \infty]$, $r \in (0,1]$, $(n,r) \neq (\infty,1)$, let 
$\mu_{n,r}$ be the random measure on $[0, 2 \pi)$
defined by 
$$ \mu_{n,r} (d \vartheta) =  (V_{n}(r)/2)^{1/2} e^{- V_{n}(r)/2}     e^{ 2 \Re \left( \sum_{1 \leq \ell \leq n} 
\frac{(r e^{i \vartheta})^\ell}{\sqrt{\ell}} \mathcal{N}_{\ell} \right)} \frac{d \vartheta}{2 \pi},$$
for 
$$V_n(r) := \operatorname{Var} \left(2 \Re \left( \sum_{1 \leq \ell \leq n} 
\frac{r^\ell}{\sqrt{\ell}} \mathcal{N}_{\ell}  \right) \right) =
2 \sum_{1 \leq \ell \leq n} \frac{r^{2\ell}}{\ell}. $$
Then, the family $(\mu_{n,r})_{n \in (1, \infty], 
r \in (0,1], (n, r) \neq (\infty,1)}$ of random measures converges in probability to the limiting random measure $\mathrm{GMC}_1$
defined in the introduction,
when $\min(n, (1-r^2)^{-1})$ tends to infinity.

\end{proposition}
 \begin{remark}
 The definition of $\mathrm{GMC}_1$ in the   introduction corresponds to Proposition \ref{GMCcritical} in the particular case $n = \infty$. 
 \end{remark}
\begin{proof}
We start the proof of Proposition \ref{GMCcritical} by the following lemma, 
estimating the covariance of the fields which are considered: 

\begin{lemma} \label{estimatecovariance}
For $n \in [1,\infty]$, $r \in (0,1]$, $(n, r) \neq (\infty, 1)$,  $\vartheta, \vartheta' \in \mathbb{R}$,  let 
$$C_n(r, \vartheta, \vartheta') 
:= \operatorname{Cov} \left(   2 \Re \left( \sum_{1 \leq \ell \leq n} 
\frac{(r e^{i \vartheta})^\ell}{\sqrt{\ell}} \mathcal{N}_{\ell} \right),  2 \Re \left( \sum_{1 \leq \ell \leq n} 
\frac{(r e^{i \vartheta'})^\ell}{\sqrt{\ell}} \mathcal{N}_{\ell} \right)\right).$$
Then 
$$C_n (r, \vartheta, \vartheta') 
=2 \sum_{1 \leq \ell \leq n} \frac{r^{2 \ell}}{\ell}
 \cos( \ell(\vartheta - \vartheta')).$$
Moreover, we have the estimates: 
$$ C_n (r, \vartheta, \vartheta') 
=2  \min ( \log n, \log ( (1 -r^2)^{-1}), \log ( ||\vartheta - \vartheta'||^{-1}) ) + \mathcal{O}(1),$$
and 
$$C_n (r, \vartheta, \vartheta') 
=- 2 \log | e^{i \vartheta }- e^{i \vartheta'}| 
+ \mathcal{O} ( ( n^{-1} + (1-r^2)) ||\vartheta - \vartheta'||^{-1}). 
$$
where $||x||$ denotes the distance from $x$ to $2 \pi \mathbb{Z}$, and where $n^{-1} := 0$ for $n = \infty$. 
 \end{lemma} 
\begin{proof}
The exact formulas are easily obtained by direct computation: for $n = \infty$, $r \in (0,1)$, 
notice that the series of normal variables are converging in $L^2$, which justifies to write the infinite sums. 
Now, we observe that 
for 
$$S( p) = \sum_{1 \leq \ell \leq p} r^{2 \ell} 
\cos (\ell ( \vartheta- \vartheta') ) 
= \Re \sum_{1 \leq \ell \leq p} (r^2 e^{i ( \vartheta- \vartheta')})^{\ell},
$$
and for $n < n' < \infty$, 
$$\sum_{n < \ell  \leq n'} \frac{r^{2 \ell}}{\ell}
 \cos( \ell(\vartheta - \vartheta'))
 = \int_{(n, n']} \frac{d S(p)}{p} 
 = \frac{S(n')}{n'} - \frac{S(n)}{n} 
 + \int_n^{n'} \frac{S(p)}{p^2} dp.$$
Now, summing the geometric series, we get 
$$|S(p)| \leq \frac{2}{ |1 - r^2 e^{i (\vartheta - \vartheta')}|},$$
where we observe that
$$|1 - r^2 e^{i (\vartheta - \vartheta')}| \geq 1 - r^2.$$
 Moreover, for $r \in (0,1/2] $, 
$$ |1 - r^2 e^{i (\vartheta - \vartheta')}| \geq 1/2 \geq ||\vartheta - \vartheta'||/(2 \pi) $$
and for $r \in (1/2, 1]$, 
$$|1 - r^2 e^{i (\vartheta - \vartheta')}| \geq |r^{-2} - e^{i (\vartheta - \vartheta')}| /4  \geq |1 -e^{i (\vartheta - \vartheta')}|/4 \geq ||\vartheta - \vartheta'||/(2 \pi).$$
 Hence, 
$$|S(p)| \leq 4 \pi \min ( (1-r^2)^{-1}, ||\vartheta - \vartheta'||^{-1})$$
and 
$$\left|\sum_{n < \ell  \leq n'} \frac{r^{2 \ell}}{\ell}
 \cos( \ell(\vartheta - \vartheta')) \right| \leq 
 12 \pi n^{-1} \min ( (1-r^2)^{-1}, ||\vartheta - \vartheta'||^{-1}). $$
If we do not have $  r = 1$ and $\vartheta \equiv \vartheta'$ modulo $2 \pi$ simultaneously, and if $n < \infty$, 
we can let $n' \rightarrow \infty$, and 
by identifying the Taylor series of the logarithm, we deduce 
$$ C_n (r, \vartheta, \vartheta')  
= - 2 \Re \log (  1 - r^{2} e^{i (\vartheta - \vartheta')} )
+ \mathcal{O}\left( n^{-1} \min ( (1-r^2)^{-1}, ||\vartheta - \vartheta'||^{-1})\right). $$
If $n = \infty$, this estimate still holds, the error term being equal to zero. 
Writing the variation of the logarithm as the integral of its derivative, we deduce 
$$ | - 2 \Re \log (  1 - r^{2} e^{i (\vartheta - \vartheta')} ) 
+ 2 \Re \log (  1 -  e^{i (\vartheta - \vartheta')} )|
\leq 2 (1 - r^2) \sup_{t \in [r^2,1]}
 | 1 -t e^{i (\vartheta - \vartheta')}|^{-1}$$
which is bounded by $4 \pi (1-r^2) ||\vartheta - \vartheta'||^{-1}$ by the previous estimate. 
This provides the estimate 
\begin{align*}
C_n (r, \vartheta, \vartheta')  
& = - 2 \Re \log (  1 -   e^{i (\vartheta - \vartheta')} )
\\ & + \mathcal{O}\left( n^{-1} \min ( (1-r^2)^{-1}, ||\vartheta - \vartheta'||^{-1}) + (1-r^2) ||\vartheta - \vartheta'||^{-1} \right)
\end{align*}
which implies the second estimate of the lemma. 
If $$n \geq  \min ( (1-r^2)^{-1}, ||\vartheta - \vartheta'||^{-1}), $$
we have 
$$ C_n (r, \vartheta, \vartheta')  
= - 2 \Re \log (  1 - r^{2} e^{i (\vartheta - \vartheta')} )
+ \mathcal{O}(1),$$
which, by the estimate 
$$ |1 - r^{2} e^{i (\vartheta - \vartheta')}| \geq \max(1 - r^2, ||\vartheta - \vartheta'||/ (2 \pi)) $$
  proven above, and the estimate 
  $$ |1 - r^{2} e^{i (\vartheta - \vartheta')}|
  \leq 1 - r^2 + ||\vartheta - \vartheta'||$$
  obtained by triangle inequality, gives
   $$C_{n} (r, \vartheta, \vartheta') 
   = 2 \min ( \log ((1-r^2)^{-1}), \log (||\vartheta - \vartheta'||^{-1}))  + \mathcal{O}(1),$$
i.e. the first estimate of the lemma in the case 
$$n \geq  \min ( (1-r^2)^{-1}, ||\vartheta - \vartheta'||^{-1}).$$
In the opposite case, we observe that in the sum 
$$2 \sum_{1 \leq \ell \leq n} \frac{r^{2 \ell}}{\ell} \cos(\ell(\vartheta -\vartheta')),$$
we can replace $r^{2 \ell} \cos (\ell (\vartheta-\vartheta'))$ by $1$, changing this quantity by 
at most $\ell((1 - r^2) +|| \vartheta - \vartheta'||) $, 
which gives 
$$2 \sum_{1 \leq \ell \leq n} \frac{r^{2 \ell}}{\ell} \cos(\ell(\vartheta -\vartheta'))
= 2 \sum_{1 \leq \ell \leq n} \frac{1}{\ell} 
+ \mathcal{O} \left( \sum_{1 \leq \ell \leq n} 
\frac{\ell((1 - r^2) +|| \vartheta - \vartheta'||)}{\ell} \right).$$
Since we assume 
$$n<   \min ( (1-r^2)^{-1}, ||\vartheta - \vartheta'||^{-1}),$$
we deduce 
$$2 \sum_{1 \leq \ell \leq n} \frac{r^{2 \ell}}{\ell} \cos(\ell(\vartheta -\vartheta')) = 2 \log n + \mathcal{O}(1),$$
which corresponds to the first estimate of the lemma.

\end{proof}
We now use Lemma \ref{estimatecovariance} in order to prove Proposition \ref{GMCcritical}. In this case $r = 1$, Proposition \ref{GMCcritical}
is already proven in \cite{JS17}. 
Indeed, the Gaussian field  we obtain in the exponent of the exponential for $r = 1$ corresponds to the regularization $X_{2,n}$ in \cite{JS17}, the only difference being the absence of random constant term $2 \sqrt{\log 2} A_0$, which multiplies all the measures by an independent log-normal variable, and then does not change their properties of convergence in probability.  

 Let $\mu$ be the limit in probability of $\mu_{n, 1}$
 when $n \rightarrow \infty$. It is enough to prove convergence in probability of $\mu_{n_k, r_k}$ towards $\mu$ when $k \rightarrow \infty$, for any sequence $(n_k, r_k)_{k \geq 1}$
 such that $n_k \in [1,\infty]$, $r_k \in (0,1]$, 
 $(n_k, r_k) \neq (1, \infty)$, 
 $\min(n_k, (1-r_k^2)^{-1})$ is increasing and goes to infinity when $k \rightarrow \infty$. 
 From the estimates of Lemma \ref{estimatecovariance}, 
 it is enough to prove convergence in probability of 
 $(\nu_k)_{k \geq 1}$ towards $\mu$, where 
 $$\nu_k (d \vartheta) = 
 e^{- V_{n_k}(r_k)/2}     e^{ 2 \Re \left( \sum_{1 \leq \ell \leq n_k} 
\frac{(r_k e^{i \vartheta})^\ell}{\sqrt{\ell}} \mathcal{N}_{\ell} \right)} \rho_k( d \vartheta)$$
and
$$\rho_k (d \vartheta) = \left( \min (\log n_k, \log ((1 - r_k^2)^{-1})) \right)^{1/2} (d \vartheta/2 \pi).$$
 The esimates of  Lemma \ref{estimatecovariance}, together with the convergence of $\mu_{n,1}$ towards $\mu$,
 also imply that the covariances 
 $$(\vartheta, \vartheta') \mapsto C_{n_k} (r_k, \vartheta, \vartheta')$$
 and 
  $$(\vartheta, \vartheta') \mapsto C_{\min(n_k, (1-r_k^2)^{-1})} (1, \vartheta, \vartheta')$$
  satisfy the assumptions of Theorem 1.1. of \cite{JS17}. This ensures that 
  $\nu_k$ converges to $\mu$ in distribution. 
  We use Theorem 4.4. of \cite{JS17} to upgrade this convergence to a convergence in probability. 
With the notation of \cite{JS17}, we take for $X$ the random distribution given by 
$$X = 2 \Re \left( \sum_{\ell \geq 1} \frac{(e^{i \vartheta})^\ell}{\sqrt{\ell}} \mathcal{N}_{\ell} \right),$$
  we take the approximating fields $(X_k)_{k \geq 1}$ given by 
$$X_k = 2 \Re \left( \sum_{1 \leq \ell \leq \min(n_k, (1-r_k^2)^{-1})} \frac{(e^{i \vartheta})^\ell}{\sqrt{\ell}} \mathcal{N}_{\ell} \right),$$
and we consider the linear operators $(R_k)_{k \geq 1}$, $R_k$
being obtained by multiplying 
the $\ell$-th Fourier coefficient of the fields by $r_k^{|\ell|} \mathds{1}_{|\ell| \leq n_k}$, for all $\ell \in \mathbb{Z}$. The operator $R_k$ can be also defined as the convolution with a Poisson kernel, followed by 
the convolution with a Dirichlet kernel, which guarantees that $R_k$ maps continuous functions to continuous functions. 

One checks that Theorem 4.4. of \cite{JS17} implies the  convergence in probability of $(\nu_k)_{k \geq 1}$ towards $\mu$, provided that we prove that the conditions given in Definition 4.3. of 
\cite{JS17} are satisfied. For fixed $k \geq 1$, 
we have $R_k X_{k'}$ independent of $k'$ as soon as 
$\min(n_{k'}, (1-r_{k'}^2)^{-1}) \geq n_k$, which implies the condition 2. of \cite{JS17}, Definition 4.3. It remains to check the condition 1., i.e. 
$$|| R_k f - f ||_{\infty} \underset{k \rightarrow \infty}{\longrightarrow} 0$$
for all $(2 \pi)$-periodic H\"older continuous functions $f$, where $|| \cdot ||_{\infty}$ denotes the supremum norm. It is sufficient to prove, for $r <1$ tending to $1$
and $n \geq 1$ tending to infinity, 
$$||P_r \star f - f||_{\infty} \underset{k \rightarrow \infty}{\longrightarrow} 0,$$
 and 
 $$||S_n ( P_r \star f) -P_r \star f  ||_{\infty} \underset{k \rightarrow \infty}{\longrightarrow} 0 $$
where $P_r$ is the Poisson kernel given by 
$$P_r (\vartheta) = \sum_{\ell \in \mathbb{Z}} r^{|\ell|} e^{i \ell\vartheta},$$
where $S_n$ is the operator 
restricting the Fourier series to 
the terms of index between $-n$ and $n$, 
and where 
$$f \star g (\vartheta) = \int_0^{2 \pi} f(\vartheta') 
g(\vartheta- \vartheta') \frac{d \vartheta}{2 \pi}.$$
If $f$ is $\alpha$-H\"older for $\alpha \in (0,1)$, 
and if 
$$||f||_{\mathcal{C}^{\alpha}} 
:= \sup_{\vartheta, \vartheta' \in \mathbb{R}, \vartheta \neq \vartheta'} |f(\vartheta) - f(\vartheta')| |\vartheta - \vartheta'|^{-\alpha},$$
we have that $ ||P_r \star f - f||_{\infty}$ is bounded by 
$$ \sup_{\vartheta_0 \in [0, 2 \pi)} \int_{-\pi}^{\pi} P_r (\vartheta) | f(\vartheta_0 - \vartheta)
- f(\vartheta_0)| \frac{d \vartheta}{2 \pi} \leq 
||f||_{\mathcal{C}^{\alpha}}
 \int_{-\pi}^{\pi} P_r (\vartheta) |\vartheta|^{\alpha} d \vartheta. 
$$
where we have the estimate 
$$P_r (\vartheta) = \frac{1 - r^2}{1 - 2r  \cos \vartheta + r^2} = \frac{1 - r^2}{(1 - r)^2 
+ (\sqrt{2 r (1 - \cos \vartheta)})^2}
\leq  \frac{1 - r^2}{2(1 - r)
\sqrt{2 r (1 - \cos \vartheta)}},
$$
which is, for $r \in (1/2, 1)$ and $|\vartheta| \leq \pi$, dominated by $1/|\vartheta|$: moreover, $P_r (\vartheta)$
tends to zero when $r \rightarrow 1$ for all $\vartheta \neq 0$. 
By dominated convergence,
$$||P_r \star f - f||_\infty \underset{r \rightarrow 1}{\longrightarrow} 0. $$
On the other hand 
$$||S_n ( P_r \star f) -P_r \star f ||_{\infty}=
 ||P_r \star (S_n f  -  f) ||_{\infty} \leq 
 ||S_n f - f||_{\infty}.$$
 The equality is due to the fact that $S_n$ is the convolution with a Dirichlet kernel, and then commutes with the convolution with $P_r$. The inequality is due to the fact that for any continuous, $(2 \pi)$-periodic function $g$, and for all $\vartheta \in [0, 2\pi)$, 
 $$|P_r \star g(\vartheta)| = \left| \int_0^{2 \pi} P_r (\vartheta - \vartheta') g(\vartheta') \frac{d \vartheta'}{2 \pi} \right| 
 \leq ||g||_{\infty} \int_0^{2 \pi} P_r (\vartheta - \vartheta') \frac{d \vartheta'}{2 \pi} \leq ||g||_{\infty}.$$
It now only remains to prove
$$||S_n f - f||_{\infty} \underset{n \rightarrow \infty}{\longrightarrow} 0,$$
which is a consequence of Jackson's theorem on the uniform convergence of Fourier series of H\"older
continuous functions.

\end{proof}

\section{Splitting of the coefficients into two parts}
\label{splitting}
For $\theta = 1$, $n \geq 1$, the coefficient $c_n$ of the HMC is given, as in \cite{NPS23}, by 
\begin{equation}c_n = \sum_{m \in S_n} \prod_{k=1}^n \frac{\mathcal{N}_k^{m_k}}{ m_k! k^{m_k/2}} \label{cnexpansion}
\end{equation}
where $S_n$ is the set of compositions 
$m = (m_k)_{k \geq 1}$ of $n$, i.e. sequences of nonnegative integers such that 
$$\sum_{k \geq 1} k m_k = n.$$

We introduce an integer parameter $L \geq 2$, and 
we decompose, for integers $n \geq 1, r \geq 0$, 
$c_{n+r}$ as $$ c_{n+r} = c^{(L)}_{n,r,g}  + c^{(L)}_{n,r,b},$$
where 
$$ c^{(L)}_{n,r,g}  = \sum_{m \in S^{(L)}_{n, r,g}}
\prod_{k=1}^n \frac{\mathcal{N}_k^{m_k}}{ m_k! k^{m_k/2}} $$
and $$ c^{(L)}_{n,r,b}  = \sum_{m \in S^{(L)}_{n,r,b}}
\prod_{k=1}^n \frac{\mathcal{N}_k^{m_k}}{ m_k! k^{m_k/2}}. $$

Here, $S_{n,r,g}^{(L)}$ is the set of "good" compositions of $n+r$, defined as compositions of $n+r$ such that for an integer $K$, $1 \leq K \leq L-1$, the largest $k$ with $m_k > 0$ satisfies $m_k= 1$ and is in  interval $(Kn/L,  (K+1)n/L]$, whereas the second largest $k$ with $m_k > 0$ is at most $Kn/L$. 
The set $S_{n,r,g}^{(L)}$ of "bad" compositions is the set of compositions of $n+r$ which are not in $S_{n,r,b}^{(L)}$. 
Our goal is to prove the following propositions: 
\begin{proposition} \label{badsum}
For all $\varepsilon > 0$, $r \geq 0$,  
$$\underset{L \rightarrow \infty}{\lim \sup} \; 
\underset{n \rightarrow \infty}{\lim \sup} \; 
\mathbb{P} (|c^{(L)}_{n,r,b}| (\log n)^{1/4} \geq \varepsilon)   = 0.$$
\end{proposition}

\begin{proposition} \label{goodsum}
For a sufficiently large integer $L \geq 2$, and for any complex polynomial 
$$p : z \mapsto \sum_{r = 0}^d a_r z^r,$$
we have 
$$\sum_{r = 0}^d a_r c^{(L)}_{n,r,g}  
\underset{n \rightarrow \infty}{\longrightarrow} 
 b(L) \left(\int_0^{2 \pi} 
|p(e^{-i \vartheta})|^2 \mathrm{GMC}_1 (d \vartheta) \right)^{1/2} \mathcal{Z} $$
where $b(L)$ is deterministic, depends only on 
$L$ and tends to $1$ when $L \rightarrow \infty$. 
\end{proposition}

These two propositions imply Theorem \ref{mainmultidim}. 
Indeed, for 
$$V :=\left( \int_0^{2 \pi} 
|p(e^{-i \vartheta})|^2 \mathrm{GMC}_1 (d \vartheta) \right)^{1/2} \mathcal{Z},$$
$$W_{n, L} := \sum_{r = 0}^d a_r c^{(L)}_{n,r,g},
$$
and for a $K$-Lipschitz function $\Phi$ from $\mathbb{C}$ to $\mathbb{R}$ with modulus bounded by $C > 0$, 
\begin{align*} &  \left|  \mathbb{E} [ \Phi (X_n(p) (\log n)^{1/4}) ] - \mathbb{E} [ \Phi(V)]  \right|
\\ &   \leq |\mathbb{E} [\Phi (W_{n,L} (\log n)^{1/4}) ] - \mathbb{E} [ \Phi(b(L) V) ]| 
\\ &   + \mathbb{E} [ |\Phi(W_{n,L}(\log n)^{1/4}) - \Phi(X_n(p) (\log n)^{1/4})| ] + \mathbb{E} [| \Phi(V) - \Phi(b(L) V)|]
\\ & \leq |\mathbb{E} [\Phi(W_{n,L} (\log n)^{1/4}) ] - \mathbb{E} [ \Phi(b(L) V) ]| 
\\ &   + \mathbb{E} \left[\min\left( K \sum_{r = 0}^d 
|a_r| |c^{(L)}_{n,r,b}|(\log n)^{1/4},  2 C\right) \right] + \mathbb{E} [\min ( K |1 - b(L)| |V|, 2C)]
\end{align*}
From Proposition \ref{goodsum}, for $L \geq 2$ large enough, 
\begin{align*} & \underset{n \rightarrow \infty}{\lim \sup} 
\left|  \mathbb{E} [ \Phi (X_n(p)(\log n)^{1/4}) ] - \mathbb{E} [ \Phi(V)]  \right| 
\\ & \leq  
   \underset{n \rightarrow \infty}{\lim \sup} \, \mathbb{E} \left[\min\left( K \sum_{r = 0}^d 
|a_r| |c^{(L)}_{n,r,b}|(\log n)^{1/4},  2 C\right) \right] \\ &  + \mathbb{E} [\min ( K |1 - b(L)| |V|, 2C)]
\end{align*}
Since the left-hand side does not depend on $L$, we can take an upper limit in $L$: 
\begin{align*} & \underset{n \rightarrow \infty}{\lim \sup} 
\left|  \mathbb{E} [ \Phi (X_n(p) (\log n)^{1/4}) ] - \mathbb{E} [ \Phi(V)]  \right| 
\\ & \leq  
  \underset{L \rightarrow \infty}{\lim \sup} \,   \underset{n \rightarrow \infty}{\lim \sup} \, \mathbb{E} \left[\min\left( K \sum_{r = 0}^d 
|a_r| |c^{(L)}_{n,r,b}|(\log n)^{1/4},  2 C\right) \right] \\ &  +   \underset{L \rightarrow \infty}{\lim \sup} \,  
 \mathbb{E} [\min ( K |1 - b(L)| |V|, 2C)]
 \\ & \leq  \underset{L \rightarrow \infty}{\lim \sup} \,   \underset{n \rightarrow \infty}{\lim \sup} \, \left(2C  \, \mathbb{P} \left(  \sum_{r = 0}^d 
|a_r| |c^{(L)}_{n,r,b}| (\log n)^{1/4} \geq \varepsilon \right)  + K \varepsilon \right)
\\ &   +   \underset{L \rightarrow \infty}{\lim \sup} \,  
 \mathbb{E} [\min ( K |1 - b(L)| |V|, 2C)]
\end{align*}
for all $\varepsilon > 0$. 
From Proposition \ref{badsum}, 
$$ \underset{n \rightarrow \infty}{\lim \sup} 
\left|  \mathbb{E} [ \Phi (X_n(p) (\log n)^{1/4}) ] - \mathbb{E} [ \Phi(V)]  \right| 
\leq   K \varepsilon 
  +   \underset{L \rightarrow \infty}{\lim \sup} \,  
 \mathbb{E} [\min ( K |1 - b(L)| |V|, 2C)].$$
By dominated convergence, the upper limit in the right-hand sides is equal to zero: 
$$\underset{n \rightarrow \infty}{\lim \sup} 
\left|  \mathbb{E} [ \Phi (X_n(p) (\log n)^{1/4}) ] - \mathbb{E} [ \Phi(V)]  \right| 
\leq   K \varepsilon.$$
Since this is true for all $\varepsilon > 0$, we deduce that the last upper limit is equal to zero, which shows the convergence in distribution 
corresponding to Theorem \ref{mainmultidim}.  Moreover, 
Theorem \ref{mainmultidim} immediately 
implies Theorem \ref{main}, provided that we check that $(\mathcal{M}_1)^{-1}$
is an exponential variable of mean $\pi$. This is a consequence of the following results, summarized in \cite{L24}: 
\begin{itemize}
\item As proven in \cite{JS17}, the measure $\mathrm{GMC}_1$, which corresponds to the 
Seneta-Heyde construction of the critical multiplicative chaos, 
has the same distribution as $\sqrt{2/\pi}$ times the 
random measure $\mu'$ obtained by constructing the critical multiplicative chaos via the derivative martingale, as 
in \cite{DRSV14a}. 
\item As proven in \cite{APS19} and \cite{P21},  
the measure $2 \mu'$ has the same distribution as 
the limit of $\mu_{\gamma}/(\sqrt{2} - \gamma)$ when $\gamma$ tends to $\sqrt{2}$ from below, $\mu_{\gamma}$ being the subcritical Gaussian multiplicative chaos obtained 
from exponentiation of $\gamma$ times 
a field on the unit circle with covariance $(z,z') \mapsto - \log |z - z'|$: notice that the critical exponent $\gamma$ is $\sqrt{2}$ with this normalization. 
\item As proven in \cite{R20}, the total mass of $\mu_{\gamma}$ is the power $-\gamma^2/2$ of 
a standard exponential variable $\mathcal{E}$, divided by $\Gamma (1-\gamma^2/2)$. 
\end{itemize}
From these three properties, we deduce that $\mathcal{M}_1$ has the same distribution as the total mass of 
$\sqrt{2/\pi} \mu'$, which itself is the limit in distribution, when $\gamma$ tends to $\sqrt{2}$ from below, of $\sqrt{2/\pi} /2 $ times the total mass of $\mu_{\gamma}/(\sqrt{2} - \gamma)$, i.e. 
$$\underset{\gamma \uparrow \sqrt{2}}{\lim} 
\frac{1}{\sqrt{2 \pi}} \, \frac{\mathcal{E}^{-\gamma^2/2}}{\Gamma(1-\gamma^2/2) (\sqrt{2}- \gamma)}
= \underset{\gamma \uparrow \sqrt{2}}{\lim} 
\frac{1}{\sqrt{2 \pi}} \, \frac{\mathcal{E}^{-1}(1- \gamma^2/2)}{\sqrt{2} - \gamma} = \frac{\mathcal{E}^{-1}}{\sqrt{\pi}}. 
$$
Hence, $\mathcal{M}_1^{-1}$ has the same distribution as $\mathcal{E} \sqrt{\pi}$, i.e. an exponential variable of mean $\sqrt{\pi}$. 
We have now deduced Theorem \ref{main} and Theorem \ref{mainmultidim} from 
Propositions \ref{badsum} and \ref{goodsum}, which are proven in the  next sections.  
\section{Proof of Proposition \ref{badsum}} \label{proofbad}
We introduce the filtration $(\mathcal{F}_k)_{k \geq 0}$, where $\mathcal{F}_k$ is the $\sigma$-algebra generated by $(\mathcal{N}_j)_{1 \leq j \leq k}$. 
For $n \geq 1$, we define 
$q := \lfloor n/ (1 + \log n) \rfloor$, which implies
that $1 \leq q \leq n$. 
For all $\varepsilon> 0$, $r \geq 0$,
\begin{align*}
\mathbb{P} [ | c^{(L)}_{n,r,b} (\log n)^{1/4} | \geq \varepsilon] & = \mathbb{E} \left[ \mathbb{P} ( | c^{(L)}_{n,r,b} (\log n)^{1/4} | \geq \varepsilon  \; | \mathcal{F}_q ) \right ]
\\ & \leq \mathbb{E} \left[ \min \left(1, \varepsilon^{-2}  (\log n)^{1/2} \mathbb{E} [|c^{(L)}_{n,r,b} |^2 | \mathcal{F}_q ]\right) \right]
\end{align*}
using Markov inequality and trivial bound of $1$ for the conditional probability.  
Hence, for any $\delta > 0$,
$$\mathbb{P} [ | c^{(L)}_{n,r,b} (\log n)^{1/4} | \geq \varepsilon] \leq   \mathbb{P} ( 
(\log n)^{1/2} \mathbb{E} [|c^{(L)}_{n,r,b} |^2 | \mathcal{F}_q ] \geq \delta) + \delta \varepsilon^{-2} 
$$
It is then enough to prove, for all $\delta > 0$, 
$$ \underset{L \rightarrow \infty}{\lim \sup} \; 
\underset{n \rightarrow \infty}{\lim \sup} \;
 \mathbb{P} ( 
(\log n)^{1/2} \mathbb{E} [|c^{(L)}_{n,r,b} |^2 | \mathcal{F}_q ] \geq \delta) = 0 $$
since this implies
$$ \underset{L \rightarrow \infty}{\lim \sup} \; 
\underset{n \rightarrow \infty}{\lim \sup} \; \mathbb{P} [ |c^{(L)}_{n,r,b} (\log n)^{1/4} | \geq \varepsilon] 
\leq \delta \varepsilon^{-2}$$
for all $\varepsilon, \delta > 0$. 

The compositions of $n+r$ in $S_{n,r,b}^{(L)}$ are those satisfying one of the following 
properties: 
\begin{itemize}
\item For all $k > n/L$, $m_k = 0$. 
\item There exists 
an integer $K$, $1 \leq K \leq L-1$
such that the largest $k$ with $m_k > 0$ is 
in $(Kn/L,  (K+1)n/L]$, and either $m_k \geq 2$
for this value of $k$, or the second largest $k$ with $m_k > 0$ is in 
$(Kn/L,  (K+1)n/L]$.
\item There exists $k$ such that $m_k \geq 1$
and $n+1 \leq k \leq n+r$. 
\end{itemize}
We have $q < n/L$ for $L$ fixed and $n$ large enough. In this case,  
giving a composition in $S_{n,r,b}^{(L)}$
is equivalent to giving an integer $s$, 
$0 \leq s \leq n+r$, a composition 
in the set $S_{s,(q)}$ of sequences $(m_k)_{1 \leq k \leq q}$ of nonnegative integers such that $\sum_{k = 1}^q k m_k = s$, 
and a composition in the set $T_{n,r,s,q}^{(L)}$
of sequences $(m_k)_{k \geq 1}$ of nonnegative integers such that 
$\sum_{k \geq 1} k m_k = n+r-s$, 
$m_k = 0$ for $1 \leq k \leq q$, 
and satisfying one of the three properties listed above. 
We write 
$$c_{n,r,b}^{(L)} =   \sum_{0 \leq s \leq n+r, m  \in S_{s,(q)}, m' \in  T_{n,r,s,q}^{(L)} } 
 \prod_{k=1}^q  \frac{\mathcal{N}_k^{m_k}}{ m_k! k^{m_k/2}}
  \prod_{k=q+1}^{n+r}  \frac{\mathcal{N}_k^{m'_k}}{ m'_k! k^{m'_k/2}}
$$
Multiplying by the conjugate, we get 
a sum on $s_1, s_2, m_1, m_2, m'_1, m'_2$. 
Taking the conditional expectation with respect to $\mathcal{F}_q$ can give non-zero terms only when $m'_1 = m'_2$, which in particular implies $s_1 = s_2$. We then get 
\begin{align*} & \mathbb{E}[|c_{n,r,b}^{(L)}|^2 | \mathcal{F}_q] \\ & = \sum_{0 \leq s \leq n+r, m_1, m_2  \in S_{s,(q)}, m' \in  T_{n,r,s,q}^{(L)} } \prod_{k=1}^q  \frac{\mathcal{N}_k^{(m_1)_k}}{ (m_1)_k! k^{(m_1)_k/2}}
\prod_{\ell=1}^q  \frac{\overline{\mathcal{N}_\ell}^{(m_2)_\ell}}{ (m_2)_\ell! \ell^{(m_2)_\ell/2}} \prod_{k=q+1}^{n+r}  \frac{1}{ m'_k! k^{m'_k}}. 
\end{align*}
If we define 
\begin{equation} c_{s,q} := \sum_{m \in S_{s, (q)} } 
\prod_{k=1}^q  \frac{\mathcal{N}_k^{m_k}}{ m_k! k^{m_k/2}}, \label{csqexpansion}
\end{equation}
we get 
$$\mathbb{E}[|c_{n,r,b}^{(L)}|^2 | \mathcal{F}_q] = \sum_{s = 0}^{n+r} 
|c_{s,q}|^2  \sum_{m \in  T_{n,r,s,q}^{(L)}}
\prod_{k=q+1}^{n+r}  \frac{1}{ m_k! k^{m_k}}
$$
and then 
$$\mathbb{E}[|c_{n,r,b}^{(L)}|^2 | \mathcal{F}_q] = \sum_{s= 0}^{n+r} 
|c_{s,q}|^2 \mathbb{P}  (E_{n,r,s,q}^{(L)})
$$
where $E_{n,r,s,q}^{(L)}$ is the event that a uniform permutation $\sigma$ on $\mathfrak{S}_{n+r-s}$ has no cycle of length smaller than or equal to $q$,
and that one of the following properties is satisfied: 
\begin{itemize}
\item $\sigma$ has no cycle of length larger than $n/L$.
\item The two largest cycle lengths of $\sigma$, counted with multiplicity, are in the interval $( Kn/L, r+(K+1)n/L]$ for the same integer $K$, $1 \leq K \leq L-1$. 
\item $\sigma$ has a cycle of length larger than $n$.
\end{itemize}
For any integer $K_0$ such that $1 \leq K_0 \leq L-1$, the following holds: if the event $E_{n,r,s,q}^{(L)}$ occurs, the permutation $\sigma$ has no cycle of length at most $q$, and one of the following properties holds: 
\begin{itemize}
\item $\sigma$ has no cycle of length larger than $K_0 n/L$.
\item There exists $\ell > K_0 n/L$ such that $\sigma$
has two cycles of length $\ell$. 
\item There exist $\ell, \ell' > K_0 n/L$, $0 < \ell'- \ell \leq n/L$ such that $\sigma$ has both cycles of length $\ell$ and $\ell'$. 
\item $\sigma$ has a cycle of length larger than $n$.
\end{itemize}
Let us assume that $s \leq n/2$, which implies that 
the order $n+r-s$ of the permutations involved is at least $n/2$. 

The Feller coupling implies that 
the joint distribution of the cycle lengths of a uniform permutation  
on $\mathfrak{S}_{n+r-s}$ is 
the same as the joint distribution of the 
spacing between $1$'s
in a sequence of independent Bernoulli
variables $((\xi_j)_{1 \leq j \leq n+r-s}, 1)$, 
where $\xi_j$ has parameter $1/j$. 
Having no spacing between $1$'s of length at most $q$ implies that all Bernoulli variables of index $2$ to $q+1$ are equal to zero. Having no spacing of length larger than $K_0 n/L$ implies that each block 
of $ \lfloor K_0 n/L \rfloor$ consecutive Bernoulli variables of index between $1$ and $n+r-s$, and in particular between $n/4$ and $n/3$, has at least one variable equal to $1$. 
For $n$ large enough, $n/4 > q + 1$, so this event is independent of the event 
concerning $\xi_j$ for $2 \leq j \leq q+1$. 
For each block, the probability to have at least a Bernoulli variable equal to $1$ 
is at most the sum of the parameters of the Bernoulli variables involved, and then at most 
$( K_0 n/L)/(n/4) =  4 K_0/L \leq 1/2$ if 
$K_0 \leq L/8$.  For $L$, $K_0$ fixed and $n$ large enough, we can find a number of disjoint block at least equal to 
$$ \frac{n/3 - n/4 + \mathcal{O}(1)}{ K_0 n/L} - 1 \geq 
 \frac{L}{13 K_0} - 1 \geq \frac{L}{26 K_0}  $$
if $K_0 \leq L/26$. The probability 
to have no Bernoulli variable of index $2$ to $q+1$ equal to $1$ is a telescopic product equal to $1/(q+1)$. 
By independence, for $L, K_0, r$ fixed, $K_0 \leq L/26$, $n$ large enough and $s \leq n/2$, the probability that a uniform permutation on $\mathfrak{S}_{n+r-s}$ has no cycle of length at most $q$ or larger than $K_0n/L$ is at most 
$q^{-1} 2^{- L/26K_0} $. 

For $1 \leq \ell < \ell'$, $ \ell + \ell' \leq n+r-s$, we now consider the probability that a uniform permutation of order $n+r-s$ has at least a cycle of length $\ell$ and a cycle of length $\ell'$, and no cycle of length up to $q$. There are 
$${n+r-s \choose \ell} (\ell-1)!
= \frac{(n+r-s)!}{(n+r-s-\ell)! \ell} $$
possible cycles of length $\ell$, and for each of these cycles, 
$$ \frac{(n+r-s- \ell)!}{(n+r-s-\ell- \ell')! \ell'} $$
possible cycles of length $\ell'$ on the complement of the support of the $\ell$-cycle. We then get 
$$\frac{(n+r-s)!}{(n+r-s-\ell - \ell')! \ell \ell'}$$
pairs of cycles. For each of these pairs, 
the probability that they appear in a uniform permutation of order $n+r-s$ and that there is no cycle of length up to $q$ is at most equal to 
$  (n+r-s-\ell-\ell')!/(n+r-s)!$ times the probability that a random permutation of 
$n+r-s-\ell -\ell'$ integers has no cycle of length up to $q$: the first factor is due to the fact that the image of $\ell + \ell'$ elements by the permutation is fixed by the presence of the $\ell$-cycle and the $\ell'$-cycle which are considered. The probability that there is no cycle of length up to $q$ is bounded by $1/(q+1)$ for any order of the permutation except zero. Hence, the probability that 
there are cycles of length $\ell$ and $\ell'$ and no cycle of length at most $q$ is bounded by $1/(q \ell \ell')$, except when $\ell + \ell' = n+r-s$, in which case it is bounded by $1/(\ell \ell')$. 
We have a similar bound when we consider two cycles of equal order $\ell = \ell'$: in fact we even gain a factor $2$ because of the symmetry between the two cycles. 

Moreover, if a permutation of order $n+r-s$ has a cycle of 
length larger than $n$, the other cycles 
have length smaller than $r$, and then smaller than $q$ for $r$ fixed and $n$ large enough. 
Hence, a permutation of order $n+r-s$ with a cycle of length larger than $n$ and no cycle of length up to $q$ is necessarily a cyclic permutation of order larger than $n$: the corresponding event occurs with probability smaller than $1/n$.

Adding all the estimates above, we deduce, for 
$L$, $K_0$, $r$ fixed, $K_0 \leq L/26$, $n$ large enough and $s \leq n/2$, 
\begin{align*}\mathbb{P} (E_{n,r,s,q}^{(L)})
 & \leq q^{-1} 2^{-L/26K_0} 
+  \sum_{\ell, \ell' \geq K_0 n/L, \,
0 \leq \ell' - \ell \leq n/L  } \frac{1}{q \ell \ell'} \\ & +\sum_{\ell, \ell' \geq K_0 n/L, \, 
0 \leq \ell' - \ell \leq n/L  , \ell + \ell' = n+r-s} \frac{1}{\ell \ell'} + \frac{1}{n}.
\end{align*}
 
The first sum is bounded by 
\begin{align*} \sum_{\ell, \ell' \geq K_0 n/L, \,
0 \leq \ell' - \ell \leq n/L} \frac{1}{q \ell^2} & \leq \sum_{\ell \geq K_0 n/L} \frac{(n/L  + 1)}{q \ell^2} 
 \\ & \leq  2\sum_{\ell \geq K_0 n/L} \frac{(n/L +1)}{q \ell(\ell+1)} 
\leq 2 q^{-1}(n/L+1) (K_0 n/L)^{-1}
\end{align*}
an then 
by $4 q^{-1} K_0^{-1}$ when $n \geq L$. 
In the sum involving the condition $\ell + \ell' = n+r-s$, $\ell'$ is necessarily at least equal to $(n+r-s)/2 \geq n/4$ (recall that we assume $s \leq n/2$ here), and then 
$\ell \geq n/4 - n/L  \geq n/8$
if $L \geq 8$.  Each term $1/\ell \ell'$ is then bounded by $32/n^2$. Moreover, the average of $\ell$ and $\ell'$ is fixed and their difference is at most $n/L$, which shows that $\ell$ and $\ell'$ are in an interval of size $\mathcal{O} (1 + n/L)$ and that there is at most one choice of $\ell'$ for each choice of $\ell$. We then get $\mathcal{O} (1 + n/L)$ terms $\mathcal{O}(n^{-2})$ in the last sum, which gives for $L$, $K_0$, $r$ fixed, $K_0 \geq 1$,  $L \geq 26 K_0$, $n \geq L$ large enough, and $s \leq n/2$, 
\begin{align*}\mathbb{P} (E_{n,r,s,q}^{(L)}) 
& \ll q^{-1} 2^{-L/26K_0} 
+ q^{-1} K_0^{-1} +   (nL)^{-1} + n^{-1}
\\ & \ll q^{-1} (2^{-L/26 K_0} +  K_0^{-1})
\end{align*}
since    for $n \geq e^{K_0}$, 
$$n \geq q \log n \geq qK_0.$$
Letting $K_0$ equal to the integer part of 
$ L/ (100 \log L)$, we get for $r$ fixed,
$L \geq 2$ large enough, $n \geq L$ large enough for a given value of $L$, and $s \leq n/2$, 
$$\mathbb{P} (E_{n,r,s,q}^{(L)})  \ll q^{-1}  (\log L)/L 
$$
Bounding the probability by one in the case where $s > n/2$, we deduce 
$$\mathbb{E}[|c_{n,r,b}^{(L)}|^2 | \mathcal{F}_q] \ll (Lq)^{-1}( \log L) \sum_{s = 0}^{\infty}
|c_{s,q}|^2 
+ \sum_{s = \lfloor n/2 \rfloor + 1}^{\infty} |c_{s,q}|^2 
$$
for $r$ fixed, $L$ large enough and for $n$ large enough when $L$ is given. 
It is then sufficient to prove that  
$$ \underset{L \rightarrow \infty}{\lim \sup} \,  \underset{n \rightarrow \infty}{\lim \sup} \, \mathbb{P} \left((Lq)^{-1}( \log L)(\log n)^{1/2} \sum_{s= 0}^{\infty}
|c_{s,q}|^2  \geq \delta \right) = 0 
$$ 
and 
$$ \underset{L \rightarrow \infty}{\lim \sup} \,  \underset{n \rightarrow \infty}{\lim \sup} \, \mathbb{P} \left((\log n)^{1/2} \sum_{s = \lfloor n/2 \rfloor + 1}^{\infty} |c_{s,q}|^2\geq \delta \right) = 0 $$
for all $\delta > 0$. 
The first double upper limit is a direct consequence of the fact that the family of random variables 
$$ \left( q^{-1} (\log q)^{1/2}  \sum_{s = 0}^{\infty}
|c_{s,q}|^2 \right)_{q \geq 1}  $$
is tight, since $\log q$ is equivalent to $\log n$ when $n \rightarrow \infty$. 
This tightness is a direct consequence of Lemma 7.3. of \cite{NPS23}. 

In the second double limit, we can discard the variable $L$. We have to 
 show a convergence in probability to zero, which is implied by the corresponding  convergence in $L^1$: 
$$(\log n)^{1/2} \sum_{s= \lfloor n/2 \rfloor + 1}^{\infty} \mathbb{E} [ |c_{s,q}|^2] \underset{n \rightarrow \infty}{\longrightarrow} 0.$$
The last expectation is equal to the probability that a uniform permutation of order $s$ has no cycle of length larger than $q$, 
or equivalently, the probability 
that there is no spacing larger than $q$ between $1$'s in a sequence of independent Bernoulli random variables 
$((\xi_j)_{1 \leq j \leq s}, 1)$, 
$\mathbb{E}[\xi_j] = 1/j$. 
If this event occurs, each block of $q$ consecutive Bernoulli variables of index between $s/2$ and $s$ has at least a $1$, an event which has probability at most $2q/s$. 
We can construct $s/2q + \mathcal{O}(1)$
disjoint blocks, which shows that 
the probability to have no cycle of length larger than $q$ is at most 
$$(2q/s)^{s/2q + \mathcal{O}(1)} \leq 
(2q/s)^{s/3q}$$
if $s/q$ is large enough. 
We then get 
$$(\log n)^{1/2} \sum_{s= \lfloor n/2 \rfloor + 1}^{\infty} \mathbb{E} [ |c_{s,q}|^2] 
\leq (\log n)^{1/2} \sum_{s > n/2} (2q/s)^{s/3 q} $$
for $n$ large enough, which implies that
$s/q \geq n/2q \gg \log n$ is also large, 
and that $2q/s$ is small.  
For $n$ large enough, we have in particular $2q/s \leq e^{-30}$, and then 
\begin{align*}
(\log n)^{1/2} \sum_{s= \lfloor n/2 \rfloor + 1}^{\infty} \mathbb{E} [ |c_{s,q}|^2] & \leq (\log n)^{1/2} e^{- 5n/q} \sum_{j = 0}^{\infty} e^{-10 j/q} 
\\ & \ll (\log n)^{1/2} e^{-5 \log n} (1 - e^{-10/ q})^{-1} \ll q (\log n)^{1/2}
n^{-5} 
\\ & \ll n^{-4}
\end{align*}
which goes to zero when $n \rightarrow \infty$. 
This completes the proof of Proposition \ref{badsum}. 

\section{Martingale central limit theorem} \label{martingaleCLT}
We now start the proof of Proposition 
\ref{goodsum}. 
The definition of the "good" compositions gives the following equality:
$$\sum_{r = 0}^d 
a_r c_{n,r,g}^{(L)} = \sum_{K = 1}^{L-1} 
\sum_{Kn/L < q \leq (K+1)n/L} \frac{\mathcal{N}_q}{\sqrt{q}} 
\sum_{r =0}^d a_r c_{n+r-q, Kn/L} 
$$
where $c_{n+r-q,Kn/L} := c_{n+r-q, \lfloor Kn/L \rfloor}$. 
The right-hand side of this equality is 
a martingale with respect to the filtration
$(\mathcal{F}_k)_{k \geq 0}$ introduced in the proof of Proposition \ref{badsum}. 
We use a version of the martingale central limit theorem given in Section 3.2 of \cite{HH80}, and restated as Theorem 4.2 of \cite{NPS23}. Notice that in this theorem, the assumption 
that the filtrations involved are nested is omitted by mistake, but it is automatically satisfied in our setting because with the notation of \cite{NPS23}, Theorem 4.2, the filtration $\mathcal{F}_{n,q}$ considered here does not depend on $n$. 

The same reasoning as in \cite{NPS23} and \cite{NPSV25}
shows that Proposition \ref{goodsum} is deduced from a suitable convergence in probability of conditional variance and a $L^4$ Lindeberg-type condition. 
In order to get the Lindeberg-type condition, we split the Gaussian variables $\mathcal{N}_q$ as follows: we 
consider i.i.d. complex Gaussians
$(\mathcal{N}_{q, t})_{q \geq 1, 1 \leq t \leq q}$
with 
$$\mathbb{E} [ \mathcal{N}_{q, t}] 
=\mathbb{E} [ ( \mathcal{N}_{q, t})^2] 
 = 0, \; \mathbb{E} [ |\mathcal{N}_{q, t}|^2] 
 = 1,$$
 and we take the variables $(\mathcal{N}_q)_{q \geq 1}$ in such a way that 
$$\mathcal{N}_q := \frac{1}{\sqrt{q}} 
\sum_{t =1}^q \mathcal{N}_{q,t},$$
which implies 
$$\sum_{r = 0}^d 
a_r c_{n,r,g}^{(L)} = \sum_{K = 1}^{L-1} 
\sum_{Kn/L < q \leq (K+1)n/L} 
\sum_{t = 1}^q \frac{\mathcal{N}_{q,t}}{q} 
\sum_{r =0}^d a_r c_{n+r-q, Kn/L} 
$$
The right-hand side provides a sum of conditionally Gaussian martingale increments
with respect to the filtration generated by 
the variables $(\mathcal{N}_{q,t})_{q \geq 1, 1 \leq t \leq q}$ ordered lexicographically:
for $1 \leq s \leq q$, the $(q(q-1)/2 + s)$-th $\sigma$-algebra is generated by 
$\mathcal{N}_{q',t}$ for $1 \leq t \leq q' < q$ and $\mathcal{N}_{q,t}$ for $1 \leq t\leq s$. 

We want to apply the martingale 
central limit theorem to the martingale just above, multiplied by $(\log n)^{1/4}$. The Lindeberg-type condition is implied by the following convergence, 
for $0 \leq r \leq d$: 
\begin{equation} (\log n) \sum_{K = 1}^{L-1} 
\sum_{Kn/L < q \leq (K+1)n/L} 
\sum_{t = 1}^q \frac{1}{q^4} 
 \mathbb{E} [| c_{n+r-q, Kn/L}|^4] 
 \underset{n \rightarrow \infty}{\longrightarrow} 0. \label{LindebergL4}
 \end{equation}
Comparing the fourth powers of the expansions \eqref{cnexpansion} and \eqref{csqexpansion}, and taking the expectation, we get an expansion
of $\mathbb{E} [| c_{n+r-q, Kn/L}|^4]$
which coincides with an expansion of 
$\mathbb{E} [ |c_{n+r-q}|^4]$ from which 
nonnegative terms have beem removed. 
Hence, 
$$ \mathbb{E} [| c_{n+r-q, Kn/L}|^4] 
\leq \mathbb{E} [| c_{n+r-q}|^4]
= n+r-q + 1 \leq n + r + 1$$
where the last equality is obtained
by taking $\theta = 1$ in equation (2.8) of \cite{NPS23}. 
The left-hand side of \eqref{LindebergL4}
is then bounded by 
$$(\log n) (n+r+1) 
\sum_{q > n/L} \frac{1}{q^3}
\leq (\log n) (n+r+1) 
\left( (n/L)^{-3}  + \int_{n/L}^{\infty}
u^{-3} du \right),
$$
which tends to zero when $n \rightarrow \infty$, proving  \eqref{LindebergL4}. 
Notice that Lindeberg $L^4$ condition would have not been satisfied without 
splitting $\mathcal{N}_q$ in terms of $\mathcal{N}_{q,t}$, because $1/q^3$ would have been replaced by $1/q^2$ in the computation above: this issue didn't not appear in the subcritical phase. 

Once Lindeberg-type condition is satisfied, we see, following the proof of Corollary 4.3 of \cite{NPS23}, that 
the convergence stated in Proposition \ref{goodsum}
follows 
from the following convergence in probability: 
$$ (\log n)^{1/2}  \sum_{K=1}^{L-1} 
\sum_{Kn/L < q \leq (K+1)n/L}  \frac{1}{q}
  \left| \sum_{r = 0}^d a_r c_{n+r -q,Kn/L} \right|^2
  \underset{n \rightarrow \infty}{
\longrightarrow} 
(b(L))^2 \mathrm{GMC}_1 (p) $$
where 
$$ \mathrm{GMC}_1 (p) := 
\int_0^{2 \pi} |p(e^{-i \vartheta})|^2 \mathrm{GMC}_1 (d \vartheta)$$
and $b(L)$ is deterministic, tending to $1$ when $L \rightarrow \infty$. 

The sum in $q$ in the left-hand side of the convergence can be written, 
for $u = Kn/L$, as
$s = n - q = Lu/K - q$, 

\begin{align*} &  \frac{1}{u} 
\sum_{(L-K-1) u/K \leq s < (L-K) u/K } 
\frac{1}{L/K - s/u} \left| \sum_{r = 0}^d 
a_r c_{s + r, u} \right|^2 
\\ & = \frac{1}{u}
\sum_{s \geq 0} \psi_{L,K} (e^{-s/u})
 \left| \sum_{r = 0}^d 
a_r c_{s + r, u} \right|^2 
 \end{align*}
 where $\psi_{L,K}$ is the function from 
 $[0,1]$ to $\mathbb{R}$ given by 
 $$\psi_{L,K} (x) = 
 \frac{\mathds{1}_{e^{-(L-K)/K} < x \leq e^{-(L-K-1)/K}}}{L/K + \log x}. 
 $$
Summing convergence in probability below 
for all values of $K$ between $1$ and $L-1$ and using the fact that $\log (Kn/L)$ is equivalent of $\log n$ for $L, K$ fixed and $n \rightarrow \infty$,  one deduce that Proposition \ref{goodsum} is a consequence of the following propositions: 
\begin{proposition} \label{convergencechaospsi}
For all integers $K, L$ such that $1 \leq K \leq L-1$, 
$$u^{-1} (\log u)^{1/2}
\sum_{s \geq 0} \psi_{L,K} (e^{-s/u})
 \left| \sum_{r = 0}^d 
a_r c_{s + r, u} \right|^2 
\underset{u \rightarrow \infty}{\longrightarrow} e^{\gamma_E} \mathbb{E}[ \psi_{L,K} (e^{-X})] \mathrm{GMC}_1(p)$$
in probability, where $\gamma_E$ is 
Euler-Mascheroni constant, 
and $X$ is a positive random variable 
with density proportional to the Dickman function. 
Recall that the Dickman function $\rho$ is the unique function from $[0, \infty)$ to $\mathbb{R}_+$ which is  
equal to $1$ on $[0,1]$, 
differentiable on $[1, \infty)$ and 
satisfies the delay differential equation 
$$x \rho'(x) + \rho(x-1) = 0. $$

\end{proposition}
\begin{proposition} \label{limitDickman}
For $L \geq 1$ integer, and $X$ following the distribution with density proportional to the Dickman function, 
$$\sum_{K=1}^{L-1} 
\mathbb{E} [ \psi_{L,K} (e^{-X})] 
\underset{L \rightarrow \infty}{\longrightarrow} e^{-\gamma_E}. 
$$
\end{proposition}
It remains to prove Propositions \ref{convergencechaospsi} and \ref{limitDickman}, which is done in the two next sections. 
\section{Proof of Proposition \ref{convergencechaospsi}} \label{proofconvergencechaospsi}
Let $m \geq 0$ be an integer. 
By Parseval's formula, we have, when $c_{s+r,u} := 0$ for $s + r < 0$, 
\begin{align*} \sum_{s \in \mathbb{Z}} e^{-sm/u}
&  \left| \sum_{r = 0}^d 
a_r c_{s + r, u} \right|^2 
 = \frac{1}{2 \pi} \int_0^{2 \pi} \left| \sum_{s \in \mathbb{Z}} \sum_{r = 0}^d
a_r c_{s + r, u} e^{-sm/2u} e^{i s \vartheta}\right|^2 \frac{d \vartheta}{2 \pi}
\\ & = \frac{1}{2 \pi} \int_0^{2 \pi} \left| \sum_{s \in \mathbb{Z}} \sum_{r = 0}^d
a_r c_{s, u} e^{-(s-r)m/2u} e^{i (s-r) \vartheta}\right|^2 \frac{d \vartheta}{2 \pi}
\\ & = \frac{1}{2 \pi} \int_0^{2 \pi} \left|
\sum_{r = 0}^d  a_r (e^{m/2u} e^{-i \vartheta})^r  \sum_{s \in \mathbb{Z}} 
 c_{s, u} e^{-sm/2u} e^{is \vartheta}\right|^2 \frac{d \vartheta}{2 \pi}
 \\ & = \frac{1}{2 \pi} \int_0^{2 \pi} |p(e^{m/2u} e^{-i \vartheta})|^2 \exp \left( 2 \Re \sum_{1 \leq k \leq u}
 \frac{(e^{-m/2u} e^{i \vartheta})^k}{\sqrt{k}} \mathcal{N}_k \right)\frac{d \vartheta}{2 \pi}
 \\ & = (V_{u}(e^{-m/2u})/2)^{-1/2} e^{V_{u}(e^{-m/2u})/2} \int_0^{2 \pi} |p(e^{m/2u} e^{-i \vartheta})|^2 \mu_{u, e^{-m/2u}} (d \vartheta), 
\end{align*}
 with the notation of Proposition 
 \ref{GMCcritical}. 
Now,
\begin{align*} e^{V_{u}(e^{-m/2u})/2 }
& = \exp \left( \sum_{1 \leq \ell \leq u}
\frac{e^{-m \ell/u}}{\ell} \right)
\\ & = e^{\log u + \gamma_E + \mathcal{O}(1/u)}\exp \left( \sum_{1 \leq \ell \leq u}
\frac{e^{-m \ell/u} - 1}{\ell} \right)
\\ & = u e^{\mathcal{O}(1/u)} e^{\gamma_E} \mathbb{E}  
\left[ \exp \left(- \frac{m}{u} \sum_{1 \leq \ell \leq u} \ell Z_{\ell} \right)  \right],
\end{align*}
where $(Z_\ell)_{\ell \geq 1}$ are independent Poisson random variables, $Z_{\ell}$ with parameter 
$1/\ell$. 
From \cite{ABT03}, Section 4.2, 
we have 
that 
$$\frac{1}{u} \sum_{1 \leq \ell \leq u} \ell Z_{\ell} \underset{u \rightarrow \infty}{\longrightarrow} X$$
in distribution, where $X$ is a positive random variable with density equal to $e^{-\gamma_E} \rho$. 
We deduce 
$$(V_{u}(e^{-m/2u})/2)^{-1/2} e^{V_{u}(e^{-m/2u})/2}
u^{-1} (\log u)^{1/2}
\underset{u \rightarrow \infty}{\longrightarrow}
e^{\gamma_E} \mathbb{E} [ e^{-m X}]. $$
Hence, 
\begin{align*} & u^{-1} (\log u)^{1/2} \sum_{s \in \mathbb{Z}} e^{-sm/u}
  \left| \sum_{r = 0}^d 
a_r c_{s + r, u} \right|^2 
\\ & =  ( 1 + o_{u \rightarrow \infty}(1) ) e^{\gamma_E} \mathbb{E} \left[ 
e^{-m X} \right]
\int_0^{2 \pi} |p(e^{m/2u} e^{-i \vartheta})|^2 \mu_{u, e^{-m/2u}} (d \vartheta)
\end{align*}
when $u \rightarrow \infty$. 
Now, for $u \geq m$, 
\begin{align*} & \int_0^{2 \pi}  \left| \, |p(e^{m/2u} e^{-i \vartheta})|^2 - |p(e^{-i \vartheta})|^2 \, \right| \mu_{u, e^{-m/2u}} (d \vartheta)
\\ & \leq \sup_{|z_1|, |z_2| \leq e^{1/2}, 
|z_1 - z_2| \leq e^{m/2u}-1} \left|\,|p(z_1)|^2 - |p(z_2)|^2\, \right|
\mu_{u, e^{-m/2u}} ([0, 2\pi))
\end{align*}
When $u \rightarrow \infty$, the supremum on $z_1$ and $z_2$ tends to zero 
by uniform continuity of $p$ on compact sets, 
whereas $\mu_{u, e^{-m/2u}} ([0, 2\pi))$
converges to $\mathrm{GMC}_1([0, 2\pi))$ in probability by Proposition \ref{GMCcritical}. By Slutsky's theorem, 
$$\int_0^{2 \pi}  (|p(e^{m/2u} e^{-i \vartheta})|^2 - |p(e^{-i \vartheta})|^2) \,  \mu_{u, e^{-m/2u}} (d \vartheta)
\underset{u \rightarrow \infty}{\longrightarrow} 0$$
in probability. 
By Proposition \ref{GMCcritical} and 
continuous mapping theorem, 
$$\int_0^{2 \pi} 
|p(e^{-i \vartheta})|^2  \mu_{u, e^{-m/2u}} (d \vartheta) 
\underset{u \rightarrow \infty}{\longrightarrow}
\int_0^{2 \pi} 
|p(e^{-i \vartheta})|^2  \mathrm{GMC}_1 (d \vartheta) $$
in probability. Combining the two last convergences, we deduce 
$$\int_0^{2 \pi} |p(e^{m/2u} e^{-i \vartheta})|^2 \mu_{u, e^{-m/2u}} (d \vartheta)
\underset{u \rightarrow \infty}{\longrightarrow}
\int_0^{2 \pi} 
|p(e^{-i \vartheta})|^2  \mathrm{GMC}_1 (d \vartheta) $$
in probability, and then 
$$u^{-1} (\log u)^{1/2} \sum_{s \in \mathbb{Z}} e^{-sm/u}
  \left| \sum_{r = 0}^d 
a_r c_{s + r, u} \right|^2 
\underset{u \rightarrow \infty}{\longrightarrow}
e^{\gamma_E} \mathbb{E} \left[ 
e^{-m X} \right] \int_0^{2 \pi} 
|p(e^{-i \vartheta})|^2  \mathrm{GMC}_1 (d \vartheta), $$
which implies, by linearity, 
\begin{equation} u^{-1} (\log u)^{1/2} \sum_{s \in \mathbb{Z}} \psi(e^{-s/u})
  \left| \sum_{r = 0}^d 
a_r c_{s + r, u} \right|^2 
\underset{u \rightarrow \infty}{\longrightarrow}
e^{\gamma_E} \mathbb{E} \left[ 
\psi(e^{-X}) \right]  \mathrm{GMC}_1 (p),
\label{convpolynomials}
\end{equation}
for all polynomial functions $\psi$.

For integers $L \geq 1$, $1 \leq K \leq L-1$, 
the function $\psi_{L,K}$ is supported 
on the interval 
$(e^{-(L-K)/K}, e^{- (L-K-1)/K}]$,
which is included in $(e^{-(L-1)}, 1]$, 
and on its support, $\psi_{L,K}$
is continuous and takes values in $[0,1]$, since the denominator of the expression of $\psi_{L,K}$ is at least $L/K - (L-K)/K = 1$. 
Hence, for all $\varepsilon \in (0,1)$, we can bound $\psi_{L,K}$ from above 
and below by nonnegative continuous functions
$\psi_1$ and $\psi_2$, 
supported on $[0,2]$, bounded by $1$, 
and equal to $\psi_{L,K}$ everywhere outside 
intervals of length $\varepsilon$ around the two 
boundaries of the support of $\psi_{L,K}$. 
By Weierstrass approximation theorem, 
the function $\psi_1$ can itself be bounded 
from above by a polynomial function $\psi_+$
such that $\psi_+ - \psi_1 \leq \varepsilon$
on $[0,2]$, and $\psi_2 $ can  be bounded 
from below by a polynomial function $\psi_-$
such that $\psi_- - \psi_2 \geq - \varepsilon$
on $[0,2]$. 
We then get
$$\psi_- \leq \psi_{L,K} \leq \psi_+$$
on the interval $[0,2]$, and 
$$\int_0^2 ( \psi_+ (t)  - \psi_- (t)) dt 
= \mathcal{O}(\varepsilon).$$
From \eqref{convpolynomials} applied to $\psi_-$,
$$\mathbb{P} \left( u^{-1} (\log u)^{1/2} \sum_{s \in \mathbb{Z}} \psi_{L,K}(e^{-s/u})
  \left| \sum_{r = 0}^d 
a_r c_{s + r, u} \right|^2  
\leq e^{\gamma_E} 
\mathbb{E} [ \psi_- (e^{-X})] \mathrm{GMC}_1 (p)
 - \varepsilon \right)
$$
tends to zero when $u \rightarrow \infty$. 
Indeed the sums in $s \in \mathbb{Z}$ involved
can only have non-zero terms for $s \geq -d$, 
which for $u \geq 2d$, implies $e^{-s/u} \in [0, 2]$ and then $\psi_{L,K} (e^{-s/u}) \geq \psi_-(e^{-s/u})$. 
Now, since $X$ has density $e^{-\gamma_E} \rho$, 
\begin{align*}e^{\gamma_E} \mathbb{E} [ \psi_- (e^{-X})]  
& = e^{\gamma_E} \mathbb{E} [ \psi_{L,K} (e^{-X})] 
- \int_0^{\infty} (\psi_{L,K}(e^{-y}) - \psi_- (e^{-y}))
\rho(y) dy
\\ & \geq e^{\gamma_E} \mathbb{E} [ \psi_{L,K} (e^{-X})] - \int_0^1 (\psi_+(t) - \psi_- (t))
\rho(\log (t^{-1})) \frac{dt}{t}.
\end{align*}
Since the Dickman function has a superexponential decay at infinity, $t^{-1} \rho(\log (t^{-1}))$
is uniformly bounded for $t \in (0,1]$. 
We deduce 
$$e^{\gamma_E} \mathbb{E} [ \psi_- (e^{-X})]  
\geq e^{\gamma_E} \mathbb{E} [ \psi_{L,K} (e^{-X})] - \mathcal{O} (\varepsilon),$$
and then 
$$e^{\gamma_E} 
\mathbb{E} [ \psi_- (e^{-X})] \mathrm{GMC}_1 (p)
\geq e^{\gamma_E} 
\mathbb{E} [ \psi_{L,K} (e^{-X})] \mathrm{GMC}_1 (p) - \mathcal{O} (A \varepsilon)
$$
as soon as $\mathrm{GMC}_1 (p) \leq A$, for any $A > 0$. 
We deduce that for all $A, \varepsilon > 0$
$$ \mathbb{P} \left( u^{-1} (\log u)^{1/2}   \sum_{s \in \mathbb{Z}}  \psi_{L,K}(e^{-s/u})  
\left| \sum_{r = 0}^d 
a_r c_{s + r, u} \right|^2  
 \leq e^{\gamma_E} 
\mathbb{E} [ \psi_{L,K} (e^{-X})] \mathrm{GMC}_1 (p)
 - \mathcal{O} ((1 + A) \varepsilon ) \right)$$
has an upper limit at most $\mathbb{P} (\mathrm{GMC}_1(p) \geq A)$ when $u \rightarrow \infty$. 
For any $\delta > 0$, we can apply this result 
to $\varepsilon \in (0, \delta/(1+A))$, and then 
let $A \rightarrow \infty$, which gives that
$$\mathbb{P} \left( u^{-1} (\log u)^{1/2} \sum_{s \in \mathbb{Z}} \psi_{L,K}(e^{-s/u})
  \left| \sum_{r = 0}^d 
a_r c_{s + r, u} \right|^2  
\leq e^{\gamma_E} 
\mathbb{E} [ \psi_{L,K} (e^{-X})] \mathrm{GMC}_1 (p)
 - \mathcal{O} (\delta ) \right)
$$
tends to zero when $u \rightarrow \infty$.
A similar reasoning involving $\psi_+$ instead of $\psi_-$ implies that 
$$\mathbb{P} \left( u^{-1} (\log u)^{1/2} \sum_{s \in \mathbb{Z}} \psi_{L,K}(e^{-s/u})
  \left| \sum_{r = 0}^d 
a_r c_{s + r, u} \right|^2  
\geq e^{\gamma_E} 
\mathbb{E} [ \psi_{L,K} (e^{-X})] \mathrm{GMC}_1 (p)
 + \mathcal{O} (\delta ) \right)
$$
tends to zero when $u \rightarrow \infty$.
Since $\psi_{L,K}$ is supported 
in $[0,1]$, we can restrict the sums in $s \in \mathbb{Z}$ to nonnegative values of $s$. 
Hence, the two convergences to zero just above together imply 
the convergence in probability stated in 
Proposition \ref{convergencechaospsi}.

\section{Proof of Proposition \ref{limitDickman}}
\label{prooflimitDickman}
From the definition of $\psi_{L,K}$, we have to prove 
$$\sum_{K=1}^{L-1} \mathbb{E} \left[ 
\frac{\mathds{1}_{(L-K-1)/K \leq X < (L-K)/K}}{L/K - X} \right] \underset{L \rightarrow \infty}{\longrightarrow} e^{-\gamma_E}, 
$$
and since $X$ has density $e^{-\gamma_E} \rho$,
we have to prove 
$$\sum_{K=1}^{L-1} \int_{(L-K-1)/K}^{(L-K)/K} 
\frac{\rho(x)}{L/K - x} dx  \underset{L \rightarrow \infty}{\longrightarrow} 1,$$
and then, letting $x = (L-K-t)/K$, 
$$\sum_{K=1}^{L-1} 
\int_0^1 \frac{ \rho((L-K-t)/K)}{ K+t} dt
\underset{L \rightarrow \infty}{\longrightarrow} 1.$$
Since $\rho$ is nonincreasing, 
$$\rho \left( \frac{L-(K-1+t)}{K-1 + t} \right) 
\leq \rho \left( \frac{L-(K+t)}{K-1 + t} \right) 
\leq \rho \left( \frac{L-(K+t)}{K} \right) 
\leq \rho \left( \frac{L-(K+t)}{K+t} \right) $$
and then, letting $u = K-1+t$ for the 
lower bound and $u = K+t$ for the upper bound, 
$$\int_0^{L-1} \frac{\rho ((L - u)/u)}{u+1}du \leq \sum_{K=1}^{L-1} 
\int_0^1 \frac{ \rho((L-K-t)/K)}{ K+t} dt
\leq \int_1^L \frac{\rho ((L - u)/u)}{u} du.$$
Letting $u = Lv$, it is then sufficient to prove: 
$$\int_{1/L}^1 \frac{\rho( (1-v)/v)}{v} dv
\underset{L \rightarrow \infty}{\longrightarrow} 1$$
and 
$$\int_{0}^{1-1/L} \frac{\rho( (1-v)/v)}{v + 1/L} dv
\underset{L \rightarrow \infty}{\longrightarrow} 1.$$
By dominated convergence, it is enough to show
$$\int_0^1 \frac{\rho( (1-v)/v)}{v} dv = 1.$$
Now, letting $x = (1-v)/v$, $v = 1/(1+x)$, 
$dv = -dx/(1+x)^2$, we get 
$$\int_0^1 \frac{\rho( (1-v)/v)}{v} dv 
= \int_0^{\infty} \frac{\rho(x)}{1 + x} dx,$$
which, by the delay differential equation satisfied by $\rho$, is equal to 
$$\int_0^{\infty} (- \rho'(x+1)) dx
= \rho(1) = 1. $$
This completes the proof of Proposition \ref{limitDickman}, which is the last step in our proof of Theorems \ref{main} and \ref{mainmultidim}.

\section{Asymptotics for moments and tails}
\label{asymptotics}
The convergence in distribution provided by Theorem \ref{main}, 
combined with a uniform bounds on moments of $c_n$ provided by \cite{SZ22}, immediately implies the following asymptotics for these moments:
\begin{cor}
    For any fixed $q \in (0,2)$, 
    $$\mathbb{E}((\log n)^{1/4}|c_n|)^{q}) \underset{n \rightarrow \infty}{\longrightarrow} 
 \pi^{-q/4}    \frac{(\pi q/2)}{\sin(\pi q/2)}.$$
\end{cor}
\begin{proof}
    By Theorem 2.1 of \cite{SZ22}, the sequence
    $$\left\{((\log n)^{1/4}|c_n|)^{q}\right\}_{n \geq 2}$$
    is uniformly integrable for $q \in (0,2)$. Hence, the convergence in distribution provided by Theorem \ref{main} implies the convergence 
    of the $q$-th moment. 
   Keeping the notation of Theorem \ref{main}, we deduce 
    $$\mathbb{E}[ ((\log n)^{1/4} |c_n|)^q] \underset{n \rightarrow \infty}{\longrightarrow}
    \mathbb{E}[ \mathcal{M}_1^{q/2} |\mathcal{Z}|^{q}] 
    = \mathbb{E}[ (\sqrt{\pi} \mathcal{E}_1)^{-q/2} \mathcal{E}_2^{q/2}],$$
    where $\mathcal{E}_1$ and $\mathcal{E}_2$ are two independent standard exponential variables. Now, 
    $$\mathbb{E} [ \mathcal{E}_1^{-q/2} \mathcal{E}_2^{q/2} ]  
    = \Gamma (1 - q/2) \Gamma(1 + q/2)
    = (q/2) \Gamma(q/2) \Gamma(1-q/2),$$
  and then by Euler's reflection formula, 
  $$\mathbb{E} [ \mathcal{E}_1^{-q/2} \mathcal{E}_2^{q/2} ]  = \frac{\pi q/2}{ \sin (\pi q/2)}.$$
\end{proof}
The convergence, stated in Theorem \ref{main}, of 
of $\sqrt{\pi} |c_n|^2 (\log n)^{1/2}$
towards a random variable of density
$x \mapsto (1+x)^2$ immediately implies
the following aymptotics on tails 
of $|c_n|$: 
\begin{cor}
For fixed $y \geq 0$, 
$$\mathbb{P} \left( (\log n)^{1/4} |c_n| \geq y \right) \underset{n \rightarrow \infty}{\longrightarrow} 
\frac{1}{1 + y^2 \sqrt{\pi}}. $$
\end{cor}
This exact asymptotics given here is available for fixed $y$ and $n \rightarrow \infty$. 
It is much more difficult to get 
bounds which are uniform in $n$ and $y$. 
Form \cite{SZ22}, one gets the following upper bound, which is sharp up to a logarithmic factor: 
\begin{proposition}
Uniformly in $y \geq 2$ and $n \geq 1$, 
$$\mathbb{P} \left( (\log n)^{1/4} |c_n| \geq y \right) = \mathcal{O} \left( \frac{\min( \sqrt{\log n}, \log y)}{y^2} \right).$$
\end{proposition}
\begin{proof}
For $n = 1$, the two sides of the estimate vanish since $\log n = 0$, so we can assume $n \geq 2$. 
  We use Markov's inequality, and Theorem 2.1 of \cite{SZ22}. We get, for $0 \leq q \leq 1$, 

    $$\mathbb{P}\left((\log n)^{1/4}|c_n| \geq y \right) \leq \frac{(\log n)^{q/2}\mathbb{E}(|c_n|^{2q})}{y^{2q}} \ll \frac{(\log n)^{q/2}}{y^{2q}(1+(1-q)\sqrt{\log n})^q}.$$
    We choose $q=1-1/(2 \log y)$, and deduce an upper bound dominated by
    $$\frac{1}{y^2}\frac{y^{1/\log y}(\log n)^{1/2-\frac{1}{4\log y}}}{\left(1+\frac{\sqrt{\log(n)}}{2\log y}\right)^{1-\frac{1}{2\log y}}}
    = \frac{e}{y^2 \left( \frac{1}{\sqrt{\log n}} + \frac{1}{2 \log y} \right)^{1 - \frac{1}{2 \log y}}}. $$
    For $n \geq 2$, $y \geq 2$, we have 
    $$\left(\frac{1}{\sqrt{\log n}} + \frac{1}{2 \log y} \right)^{\frac{1}{2 \log y}} \leq 
    \left(1 + \frac{1}{ \sqrt{\log 2}}+\frac{1}{2 \log 2} \right)^{\frac{1}{2 \log 2}}
    \leq 10,
    $$
    which gives an upper bound dominated by 
    $$\frac{1}{ y^2 \left(\frac{1}{\sqrt{\log n}} + \frac{1}{2 \log y} \right)} \leq \frac{1}{y^2 \max(1/\sqrt{\log n}, 1/(2 \log y))} 
    \leq \frac{\min(\sqrt{\log n}, 2\log y)}{y^2}.$$
    \end{proof}

It is much more difficult to obtain lower bounds for the tails which are uniform in $n$ and $y$. Notice that for fixed $n \geq 1$, $c_n$ is a linear combination of product of a bounded number of complex Gaussian variables, and then $(\log n)^{1/4} |c_n|$ has 
tails decaying faster than polynomially. 
In particular, it is not possible to have 
$$\mathbb{P} \left((\log n)^{1/4} |c_n| \geq y\right) \gg 1/y^2$$
uniformly in $n \geq 2$ and $y \geq 2$. 

\section*{Acknowledgments}
We thank Elliot Paquette and Nick Simm for helpful discussions on the holomorphic multiplicative chaos following our previous work on the subcritical regime. In particular, the reasoning providing the constant $\sqrt{\pi}$ in Theorem \ref{main} has been obtained from computations initially written by Nick Simm. 
We would also like to thank Seth Hardy and  Oleksiy Klurman for their interest in this paper and for their comments on the sums of random multiplicative functions analogue for the tail asymptotics. The first author is supported by the Heilbronn Institute for Mathematical Research.

\color{black}

\bibliographystyle{habbrv}
\bibliography{biblio}

\end{document}